\documentclass[12pt]{amsart}%
\usepackage{graphicx}
\usepackage{amscd, color}
\usepackage{amsmath}
\usepackage{amsfonts}
\usepackage{amssymb}%
\providecommand{\U}[1]{\protect\rule{.1in}{.1in}}
\providecommand{\U}[1]{\protect\rule{.1in}{.1in}}
\providecommand{\U}[1]{\protect\rule{.1in}{.1in}} \textwidth 16.3cm
\theoremstyle{plain}
\newtheorem{acknowledgement}{Acknowledgement}

\newtheorem{theorem}{Theorem}[section]

\newtheorem{corollary}[theorem]{Corollary}

\newtheorem{remark}[theorem]{Remark}

\newtheorem{problem}[theorem]{Problem}

\numberwithin{equation}{section}

\begin{document}
\title[On cotype and a Grothendieck-type theorem ]{On cotype and a Grothendieck-type theorem for absolutely summing multilinear operators}
\author{A. T. Bernardino}
\address[A. T. Bernardino]{UFRN/CERES - Centro de Ensino Superior do Serid\'{o}, Rua
Joaquim Greg\'{o}rio, S/N, 59300-000, Caic\'{o}- RN, Brazil}

\begin{abstract}
A famous result due to Grothendieck asserts that every continuous linear
operator from $\ell_{1}$ to $\ell_{2}$ is absolutely $(1,1)$-summing. If
$n\geq2,$ however, it is very simple to prove that every continuous $n$-linear
operator from $\ell_{1}\times\cdots\times\ell_{1}$ to $\ell_{2}$ is absolutely
$\left(  1;1,...,1\right)  $-summing, and even absolutely $\left(  \frac{2}%
{n};1,...,1\right)  $-summing$.$ In this note we deal with the following problem:

Given a positive integer $n\geq2$, what is the best constant $g_{n}>0$ so that
every $n$-linear operator from $\ell_{1}\times\cdots\times\ell_{1}$ to
$\ell_{2}$ is absolutely $\left(  g_{n};1,...,1\right)  $-summing?

We prove that $g_{n}\leq\frac{2}{n+1}$ and also obtain an optimal improvement
of previous recent results (due to Heinz Juenk $\mathit{et}$ $\mathit{al}$,
Geraldo Botelho $\mathit{et}$ $\mathit{al}$ and Dumitru Popa) on inclusion
theorems for absolutely summing multilinear operators.

\end{abstract}
\maketitle

\section{Introduction}

Grothendieck's theorem for absolutely summing operators asserts that every
continuous linear operator from $\ell_{1}$ to $\ell_{2}$ is absolutely
$(1;1)$-summing (and hence absolutely $(p;p)$-summing for every $p\geq1$). For
the linear theory of absolutely summing operators we refer to \cite{df, djt}
(see also \cite{bpr, ku2, seo} for recent developments).

In the multilinear setting, D. P\'{e}rez-Garc\'{\i}a, in his PhD thesis
\cite{dav} (see also \cite{bom} and \cite{botp} for a different proof), proved
that every continuous $n$-linear operator from $\ell_{1}\times\cdots\times
\ell_{1}$ to $\ell_{2}$ is multiple $(1;1,...,1)$-summing (in fact, multiple
$(p;p,...,p)$-summing for every $1\leq p\leq2$)$.$ This result can be regarded
as the multilinear version of Grothendieck's theorem.

Let us recall the notions.

The letters $X_{1},...,X_{n},X,Y$ will always denote Banach spaces over
$\mathbb{K}=\mathbb{R}$ or $\mathbb{C}$ and $X^{\ast}$ represents the
topological dual of $X$.

For any $s>0,$ we denote the conjugate of $s$ by $s^{\ast}$. Given a positive
integer $n$, the space of all continuous $n$-linear operators from
$X_{1}\times\cdots\times X_{n}$ to $Y$ \ endowed with the $\sup$ norm is
denoted by $\mathcal{L}(X_{1},...,X_{n};Y).$ For $p>0$, the vector space of
all sequences $\left(  x_{j}\right)  _{j=1}^{\infty}$ in $X$ such that%
\[
\left\Vert \left(  x_{j}\right)  _{j=1}^{\infty}\right\Vert _{p}=\left(
\sum_{j=1}^{\infty}\left\Vert x_{j}\right\Vert ^{p}\right)  ^{\frac{1}{p}%
}<\infty
\]
is denoted by $\ell_{p}\left(  X\right)  .$ We represent by $\ell_{p}%
^{w}\left(  X\right)  $ the linear space of the sequences $\left(
x_{j}\right)  _{j=1}^{\infty}$ in $X$ such that $\left(  \varphi\left(
x_{j}\right)  \right)  _{j=1}^{\infty}\in\ell_{p}\left(  \mathbb{K}\right)  $
for every $\varphi\in X^{\ast}$.

If $0<p,q_{1},...,q_{n}<\infty$ and $\frac{1}{p}\leq\frac{1}{q_{1}}%
+\cdots+\frac{1}{q_{n}},$ a multilinear operator $T\in\mathcal{L}%
(X_{1},...,X_{n};Y)$ is absolutely\emph{ }$(p;q_{1},...,q_{n})$-summing if
$(T(x_{j}^{(1)},...,x_{j}^{(n)}))_{j=1}^{\infty}\in\ell_{p}(Y)$ for every
$(x_{j}^{(k)})_{j=1}^{\infty}\in\ell_{q_{k}}^{w}(X_{k}),k=1,...,n.$ In this
case we write $T\in\Pi_{\left(  p;q_{1},...,q_{n}\right)  }^{n}\left(
X_{1},...,X_{n};Y\right)  $. For details we refer to \cite{am}.

When $1\leq q_{1},...,q_{n}\leq p<\infty$ a multilinear operator
$T\in\mathcal{L}(X_{1},...,X_{n};Y)$ is multiple\emph{ }$(p;q_{1},...,q_{n}%
)$-summing if $(T(x_{j_{1}}^{(1)},...,x_{j_{n}}^{(n)}))_{j_{1},..,j_{n}%
=1}^{\infty}\in\ell_{p}(Y)$ for every $(x_{j}^{(k)})_{j=1}^{\infty}\in
\ell_{q_{k}}^{w}(X_{k}),k=1,...,n.$ In this case we write $T\in\Pi_{m\left(
p;q_{1},...,q_{n}\right)  }^{n}\left(  X_{1},...,X_{n};Y\right)  $. For
details we mention \cite{bom, collec} and for recent developments and
applications related to the multilinear and polynomial theory we refer to
\cite{ag, bbb, bh, an, dgm, se, lit, mat, ppp} and references therein. For
$n=1$ we write $\Pi$ instead of $\Pi^{1}$ and we recover the classical theory
of absolutely summing linear operators.

For $1\leq q_{1},...,q_{n}\leq p<\infty,$ the inclusion
\[
\Pi_{m\left(  p;q_{1},...,q_{n}\right)  }^{n}\left(  X_{1},...,X_{n};Y\right)
\subseteqq\Pi_{\left(  p;q_{1},...,q_{n}\right)  }^{n}\left(  X_{1}%
,...,X_{n};Y\right)
\]
is obvious. So, the following coincidence result is an immediate consequence
of P\'{e}rez-Garc\'{\i}a multilinear version of Grothendieck's theorem:

\begin{theorem}
For every positive integer $n$,
\[
\Pi_{\left(  1;1,...,1\right)  }^{n}\left(  \ell_{1},...,\ell_{1};\ell
_{2}\right)  =\mathcal{L}\left(  \ell_{1},...,\ell_{1};\ell_{2}\right)  .
\]

\end{theorem}

However, using that $\ell_{1}$ has cotype $2$ it is easy to prove that the
above result is far from being optimal. In fact, we have the following
improvement (see \cite{irish, irishd}):

\begin{theorem}
For every positive integer $n\geq2$,
\begin{equation}
\Pi_{\left(  \frac{2}{n};1,...,1\right)  }^{n}\left(  \ell_{1},...,\ell
_{1};\ell_{2}\right)  =\mathcal{L}\left(  \ell_{1},...,\ell_{1};\ell
_{2}\right)  . \label{dda}%
\end{equation}

\end{theorem}

So, the following problem is quite natural:

\begin{problem}
\label{xc}Given a positive integer $n\geq2$, what is the best constant
$g_{n}>0$ so that%
\[
\Pi_{\left(  g_{n};1,...,1\right)  }^{n}\left(  \ell_{1},...,\ell_{1};\ell
_{2}\right)  =\mathcal{L}\left(  \ell_{1},...,\ell_{1};\ell_{2}\right)  ?
\]

\end{problem}

If we test $n=1$ in (\ref{dda}) we obtain
\[
\Pi_{(2;1)}\left(  \ell_{1};\ell_{2}\right)  =\mathcal{L}\left(  \ell_{1}%
;\ell_{2}\right)
\]
which is not surprising at all, in view of Grothendieck's Theorem. So, in some
sense, we feel that the estimate $g_{n}\leq\frac{2}{n}$ for $n\geq2$ is
probably not optimal. The optimistic reader will probably hope for an estimate
for $g_{n}$ so that in the case $n=1$ we recover Grothendieck's Theorem.
Fortunately, in the last section we will precisely obtain such an estimate.

The problem of estimating $g_{n}$ is related to the generalization of certain
results involving cotype and absolutely summing multilinear operators. The
following result is a combination of \cite[Theorem $3$ and Remark $2$]{junek},
\cite[Corollary $4.6$]{pop} and \cite[Theorem $3.8$ (ii)]{michels}:

\begin{theorem}
[Inclusion Theorem]\label{tt} Let $X_{1},...,X_{n}$ be Banach spaces with
cotype $s$ and $n\geq2$ be a positive integer:

(i) If $s=$ $2,$ then
\begin{equation}
\Pi_{\left(  q;q,...,q\right)  }^{n}(X_{1},...,X_{n};Y)\subseteqq\Pi_{\left(
p;p,...,p\right)  }^{n}(X_{1},...,X_{n};Y) \label{uu}%
\end{equation}
holds true for $1\leq p\leq q\leq2$ and every $Y$.

(ii) If $s>2,$ then
\begin{equation}
\Pi_{\left(  q;q,...,q\right)  }^{n}(X_{1},...,X_{n};Y)\subseteqq\Pi_{\left(
p;p,...,p\right)  }^{n}(X_{1},...,X_{n};Y) \label{vv}%
\end{equation}
holds true for $1\leq p\leq q<s^{\ast}$ and every $Y$.
\end{theorem}

The results above are clearly not always optimal since, for example,%
\[
\Pi_{\left(  2;2,2,2\right)  }^{3}(\ell_{2},\ell_{2},\ell_{2};\mathbb{K}%
)\neq\mathcal{L}(\ell_{2},\ell_{2},\ell_{2};\mathbb{K})=\Pi_{\left(  \frac
{2}{3};1,1,1\right)  }^{3}(\ell_{2},\ell_{2},\ell_{2};\mathbb{K}).
\]

So, another natural problem is:

\begin{problem}
\label{xz} Given $1\leq p\leq q<\infty$ and a positive integer $n\geq2$, what
are the optimal $\alpha:=\alpha_{p,q,n}>0$ so that, under the same
circumstances of (\ref{uu}) and (\ref{vv}), we have%
\begin{equation}
\Pi_{\left(  q;q,...,q\right)  }^{n}(X_{1},...,X_{n};Y)\subseteqq\Pi_{\left(
\alpha;p,...,p\right)  }^{n}(X_{1},...,X_{n};Y) \label{ii}%
\end{equation}
for all Banach spaces $X_{1},...,X_{n},Y$ $?$
\end{problem}

In this direction we extend Theorem \ref{tt} and also recent results from
\cite{4, 3} by showing that
\[
\alpha\leq\frac{qp}{n\left(  q-p\right)  +p}%
\]
and, in some sense, this constant is optimal, since for this value of $\alpha$
we have an equality in (\ref{ii}).

\section{An estimate for $\alpha$}

\begin{theorem}
Let $1\leq k\leq n,$ where $n\geq2$ is a positive integer. If $X_{i}$ has
cotype $s_{i}\geq2,i=1,...,k$ and
\[
1\leq p\leq q<\min_{1\leq i\leq k}s_{i}^{\ast}\text{ if }s_{i}>2\text{ for
some }i=1,...,k
\]
or
\[
1\leq p\leq q\leq2\text{ if }s_{i}=2\text{ for all }i=1,...,k,
\]
then
\[
\Pi_{(z;q,...,q,t,...,t)}^{n}(X_{1},...,X_{n};Y)=\Pi_{(\frac{zqp}{zk\left(
q-p\right)  +qp};p,...,p,t,...,t)}^{n}(X_{1},...,X_{n};Y),
\]
for all $X_{k+1},...,X_{n},Y$ and all $z,t\geq1$ (here $q$ and $p$ are
repeated $k$ times). In particular, if $k=n$,%
\[
\Pi_{(z;q,...,q)}^{n}(X_{1},...,X_{n};Y)=\Pi_{(\frac{zqp}{zk\left(
q-p\right)  +qp};p,...,p)}^{n}(X_{1},...,X_{n};Y)
\]

\end{theorem}

\begin{proof}
Since $X_{i}$ has finite cotype $s_{i}\geq2,i=1,...,k$, then we have%
\[
\ell_{p}^{w}(X_{i})=\ell_{qp/\left(  q-p\right)  }\ell_{q}^{w}(X_{i})
\]

for all $i=1,...,k$ with
\[
1\leq p\leq q<\min_{1\leq i\leq k}s_{i}^{\ast}\text{ if }s_{i}>2\text{ for
some }i=1,...,k
\]
or
\[
1\leq p\leq q\leq2\text{ if }s_{i}=2\text{ for all }i=1,...,k.
\]

Let $(x_{j}^{(i)})_{j=1}^{\infty}\in\ell_{p}^{w}(X_{i}),i=1,...,k$ and
$(x_{j}^{(i)})_{j=1}^{\infty}\in\ell_{t}^{w}(X_{i})$ for $i=k+1,...,n$. So
$x_{j}^{(i)}=\alpha_{j}^{\left(  i\right)  }y_{j}^{\left(  i\right)  },$ with
$\left(  \alpha_{j}^{\left(  i\right)  }\right)  _{j=1}^{\infty}\in
\ell_{qp/\left(  q-p\right)  }$ and $\left(  y_{j}^{\left(  i\right)
}\right)  _{j=1}^{\infty}\in\ell_{q}^{w}\left(  X_{i}\right)  ,$ for all $j$
and $i=1,...,k$. If $A\in\Pi_{(z;q...,q,t,...,t)}^{n}(X_{1},...,X_{n};Y)$,
then%
\begin{align*}
&  \left(  \sum_{j=1}^{\infty}\left\Vert A\left(  x_{j}^{(1)},...,x_{j}%
^{\left(  n\right)  }\right)  \right\Vert ^{\frac{zqp}{zk\left(  q-p\right)
+qp}}\right)  ^{\frac{zk\left(  q-p\right)  +qp}{zqp}}\\
&  =\left(  \sum_{j=1}^{\infty}\left(  \left\vert \alpha_{j}^{\left(
1\right)  }\cdots\alpha_{j}^{\left(  k\right)  }\right\vert \left\Vert
A\left(  y_{j}^{(1)},...,y_{j}^{\left(  k\right)  },x_{j}^{(k+1)}%
,...,x_{j}^{\left(  n\right)  }\right)  \right\Vert \right)  ^{\frac
{zqp}{zk\left(  q-p\right)  +qp}}\right)  ^{\frac{zk\left(  q-p\right)
+qp}{zqp}}\\
&  \leq\left(  \sum_{j=1}^{\infty}\left\Vert A\left(  y_{j}^{(1)}%
,...,y_{j}^{\left(  k\right)  },x_{j}^{(k+1)},...,x_{j}^{\left(  n\right)
}\right)  \right\Vert ^{z}\right)  ^{\frac{1}{z}}\left(  \sum_{j=1}^{\infty
}\left\vert \alpha_{j}^{\left(  1\right)  }\cdots\alpha_{j}^{\left(  k\right)
}\right\vert ^{\frac{qp}{k\left(  q-p\right)  }}\right)  ^{k\left(  \frac
{q-p}{qp}\right)  }\\
&  \leq\left(  \sum_{j=1}^{\infty}\left\Vert A\left(  y_{j}^{(1)}%
,...,y_{j}^{\left(  k\right)  },x_{j}^{(k+1)},...,x_{j}^{\left(  n\right)
}\right)  \right\Vert ^{z}\right)  ^{\frac{1}{z}}%
{\displaystyle\prod\limits_{i=1}^{k}}
\left(  \sum_{j=1}^{\infty}\left\vert \alpha_{j}^{\left(  i\right)
}\right\vert ^{\frac{qp}{\left(  q-p\right)  }}\right)  ^{\frac{q-p}{qp}%
}<\infty
\end{align*}
and we conclude that%
\[
\Pi_{(z;q,...,q,t,...,t)}^{n}(X_{1},...,X_{n};Y)\subseteqq\Pi_{(\frac
{zqp}{zk\left(  q-p\right)  +qp};p,...,p,t,...,t)}^{n}(X_{1},...,X_{n};Y).
\]
The other inclusion is a consequence of the inclusion theorem for absolutely
summing multilinear operators.
\end{proof}

A similar result holds if $X_{j_{1}},...,X_{j_{k}},\left\{  j_{1}%
,...,j_{k}\right\}  \subseteqq\left\{  1,...,n\right\}  $ (instead of
$X_{1},...,X_{k}$) have cotype $s_{j_{i}}\geq2,i=1,...,k.$

The following immediate corollary is an optimal (in the sense that we have an
equality instead of an inclusion) generalization of Theorem \ref{tt}:

\begin{corollary}
\label{ut}If $n\geq2$ and $X_{1},...,X_{n}$ have finite cotype $s$ and
\[
1\leq p\leq q<s^{\ast}\text{ if }s>2
\]
or%
\[
1\leq p\leq q\leq2\text{ if }s=2,
\]
then%
\[
\Pi_{(q;q,...,q)}^{n}(X_{1},...,X_{n};Y)=\Pi_{(\frac{qp}{n\left(  q-p\right)
+p};p,...,p)}^{n}(X_{1},...,X_{n};Y)
\]
for every Banach space $Y$ and%
\[
\alpha\leq\frac{qp}{n\left(  q-p\right)  +p}.
\]

\end{corollary}

\begin{remark}
The above results were independently proved in \cite{Bla}.
\end{remark}

\section{An estimate for $g_{n}$}

From Corollary \ref{ut} we know that
\[
\Pi_{(2;2,...,2)}^{n}\left(  \ell_{1},...,\ell_{1};\ell_{2}\right)
=\Pi_{(\frac{2}{n+1};1,...,1)}^{n}\left(  \ell_{1},...,\ell_{1};\ell
_{2}\right)
\]
for all $n\geq2$. But, since%
\[
\mathcal{L}\left(  \ell_{1},...,\ell_{1};\ell_{2}\right)  =\Pi_{m(2;2,...,2)}%
^{n}\left(  \ell_{1},...,\ell_{1};\ell_{2}\right)  \subseteqq\Pi
_{(2;2,...,2)}^{n}\left(  \ell_{1},...,\ell_{1};\ell_{2}\right)
\]
it readily follows that
\[
\Pi_{(\frac{2}{n+1};1,...,1)}^{n}\left(  \ell_{1},...,\ell_{1};\ell
_{2}\right)  =\mathcal{L}\left(  \ell_{1},...,\ell_{1};\ell_{2}\right)
\]
for all $n\geq2$. So we have:

\begin{theorem}
If $n\geq2$, then%
\[
g_{n}\leq\frac{2}{n+1}.
\]

\end{theorem}

Note that Grothendieck's Theorem asserts that $g_{1}=1$ and $1=\frac{2}{1+1};$
hence we conjecture that $\frac{2}{n+1}$ is in fact the optimal estimate for
$g_{n}$.

\bigskip

\begin{acknowledgement}
This paper is a part of the doctoral thesis of the author which is being
written under supervision of Prof. Daniel Pellegrino. The author thanks Prof.
Pellegrino for introducing him to the subject and the main problem from this
note and also for several suggestions and important insights.
\end{acknowledgement}

\end{document}